\documentclass[]{amsart}

\usepackage{graphicx,array}
\usepackage{amsmath}
\usepackage{amscd}
\usepackage{amsfonts}
\usepackage{amssymb}
\usepackage{hyperref}
\usepackage{amsthm,url}
\usepackage{mathabx}
\usepackage{multimedia}
\usepackage{algorithm}
\usepackage{algpseudocode}

\theoremstyle{plain}

\newtheorem{Rem}{Remark}
\newtheorem{Pro}{Proposition}
\theoremstyle{definition}

\usepackage{xspace}

\newcommand{\LtR}{ L^2 (\mathbb{R})}

\def\RR{\mathbb{R}}

\def\CC{\mathbb{C}}
\def\ZZ{\mathbb{Z}}
\def\NN{\mathbb{N}}

\def\OO{\mathcal{O}}

\newcommand{\bd}{\mathbf}

\begin{document}
%
\title{A framework for invertible, real-time
constant-Q transforms }
%
%
%

\author[N.~Holighaus]{Nicki Holighaus}
\address{Acoustics Research Institute, Austrian Academy of Sciences, Wohllebengasse 12-14, 1040 Vienna, Austria}
\email[Nicki Holighaus]{nicki.holighaus@univie.ac.at}

\author[M.~D\"orfler]{Monika D\"orfler}
\address{Numerical Harmonic Analysis Group, Faculty of Mathematics, University of Vienna, Alserbachstra\ss{}e 23, 1090 Vienna, Austria}
\email[Monika D\"orfler]{monika.doerfler@univie.ac.at}

\author[G.~Velasco]{Gino Angelo Velasco}
\address{Institute of Mathematics, University of the Philippines-Diliman, 1101 Quezon City, Philippines}
\email[Gino Angelo Velasco]{gamvelasco@math.upd.edu.ph}

\author[T.~Grill]{Thomas Grill}
\address{Austrian Research Institute for Artificial Intelligence, Freyung 6/6, 1010 Vienna, Austria}
\email[Thomas Grill]{gr@grrrr.org}

\date{\today}

\markboth{sliCQ manuscript}%
{Revision}
%

\begin{abstract}
Audio signal processing frequently requires time-frequency representations and in many applications,
a non-linear spacing of frequency-bands is preferable. This paper introduces a framework for efficient
implementation of  invertible signal transforms allowing for non-uniform and
in particular non-linear frequency resolution. Non-uniformity in frequency
is realized by applying  \emph{nonstationary Gabor frames} with adaptivity
in the frequency domain.
The realization of a perfectly invertible  \emph{constant-Q} transform is described
in detail.  To achieve real-time processing, independent of signal length,
slice-wise processing of the full input signal is proposed and referred to
as \emph{sliCQ transform}. 

By applying frame theory and FFT-based processing, the presented approach overcomes computational inefficiency 
and lack of invertibility of classical constant-Q transform implementations. Numerical simulations evaluate the  efficiency of the proposed algorithm and 
the method's applicability is illustrated by experiments on   real-life audio 
signals.
\end{abstract}

\maketitle



%

\section{Introduction}
\label{sec:intro}
Analysis, synthesis and processing of sound  is commonly based on the representation  of audio signals by means of time-frequency dictionaries.   
The short-time Fourier transform (STFT), also referred to as \emph{Gabor transform}, is a widely
used tool 
due to its straight-forward
interpretation and FFT-based implementation, which ensure efficiency and invertibility \cite{ma09,do86}. 
STFT features a uniform time and frequency resolution and a linear spacing of the time frequency bins.

In contrast, the constant-Q transform (CQT), 
originally introduced in \cite{boyo78} and in music processing by J.~Brown~\cite{BR91}, provides a frequency resolution that depends on geometrically spaced center 
frequencies of the analysis windows. In particular, the Q-factor, i.e.~the ratio of  center frequency to bandwidth of each window, is constant over all frequency  bins; the constant Q-factor leads to a finer frequency resolution 
in  low frequencies  whereas time resolution improves with increasing frequency. This principle makes the constant-Q transform well-suited for audio data, since it better reflects the resolution of the human auditory system than the linear frequency-spacing provided by the FFT, cf.~\cite{SM09} and references therein. Furthermore, musical characteristics such as overtone structures remain invariant under frequency shifts in a constant-Q transform, which is a natural feature from  a perception point of view. 
In speech 
and music processing, perception-based considerations are important, which is one of the reasons why CQTs, due to their previously discussed properties, are often desirable  in these fields.  An example of a CQ-transform, obtained with our algorithm, is shown in Figure~\ref{fig:STFTvsNSGT}.

The principal idea of CQT is  reminiscent of wavelet transforms, compare~\cite{BaSe09}.  As opposed to wavelet transforms, the original CQT is not invertible and does not rely on any concept of (orthonormal) bases. On the other hand, the 
number of bins (frequency channels) per octave is much higher in the CQT than most traditional wavelet techniques would allow for. Partly due to this requirement, the 
computational efficiency of the original transform as well as  its improved versions, cf.~\cite{cqt92}, may often be insufficient.
Moreover,  the lack of invertibility of existing CQTs has become an important issue: for some desired
applications, such as extraction and modification, e.g. transposition,  of distinct parts of the signal, the unbiased reconstruction from analysis coefficients is crucial. Approximate methods for reconstruction from constant-Q coefficients have been proposed before, in particular for signals which are sparse in the frequency domain~\cite{crcyfi06} and by octave-wise processing in~\cite{klsc10}.

In the present contribution,  we are  interested in  inversion in the sense of {\it perfect reconstruction}, i.e.~up to numerical precision; to this end, we  investigate a new approach to constant-Q  signal processing. The presented framework has the following core properties:
\begin{enumerate}
\item Relying on concepts from frame theory,~\cite{ma09}, we suggest  the implementation  of a constant-Q transform using the nonstationary Gabor transform (NSGT), which guarantees perfect invertibility. This perfectly invertible constant-Q transform is subsequently called \emph{constant-Q nonstationary Gabor transform} (CQ-NSGT).
\item We introduce a  preprocessing step by \emph{slicing} the signal to pieces of (usually uniform) finite length. Together with FFT-based methods, this allows for bounded delay and results in linear processing time. Thus, our algorithm   lends itself to real-time processing and the resulting transform is referred to as  \emph{sliced constant-Q transform} (sliCQ).
\end{enumerate}
NSGTs, introduced in \cite{ja05-2,badohojave11}, generalize  the classical sampled short-time Fourier transform
or Gabor transform~\cite{ma09,fest98}. They allow for fast, FFT-based implementation of both analysis and reconstruction under mild conditions on the analysis windows. The CQ-NSGT was first presented  in~\cite{dogrhove11}; the  frequency-resolution of the proposed CQ-NSGT is essentially identical to  that of the CQT, cf.~Figure~\ref{fig:STFTvsNSGT} for an example. 

\begin{figure}[t!]
\begin{center}
\includegraphics[width=8cm]{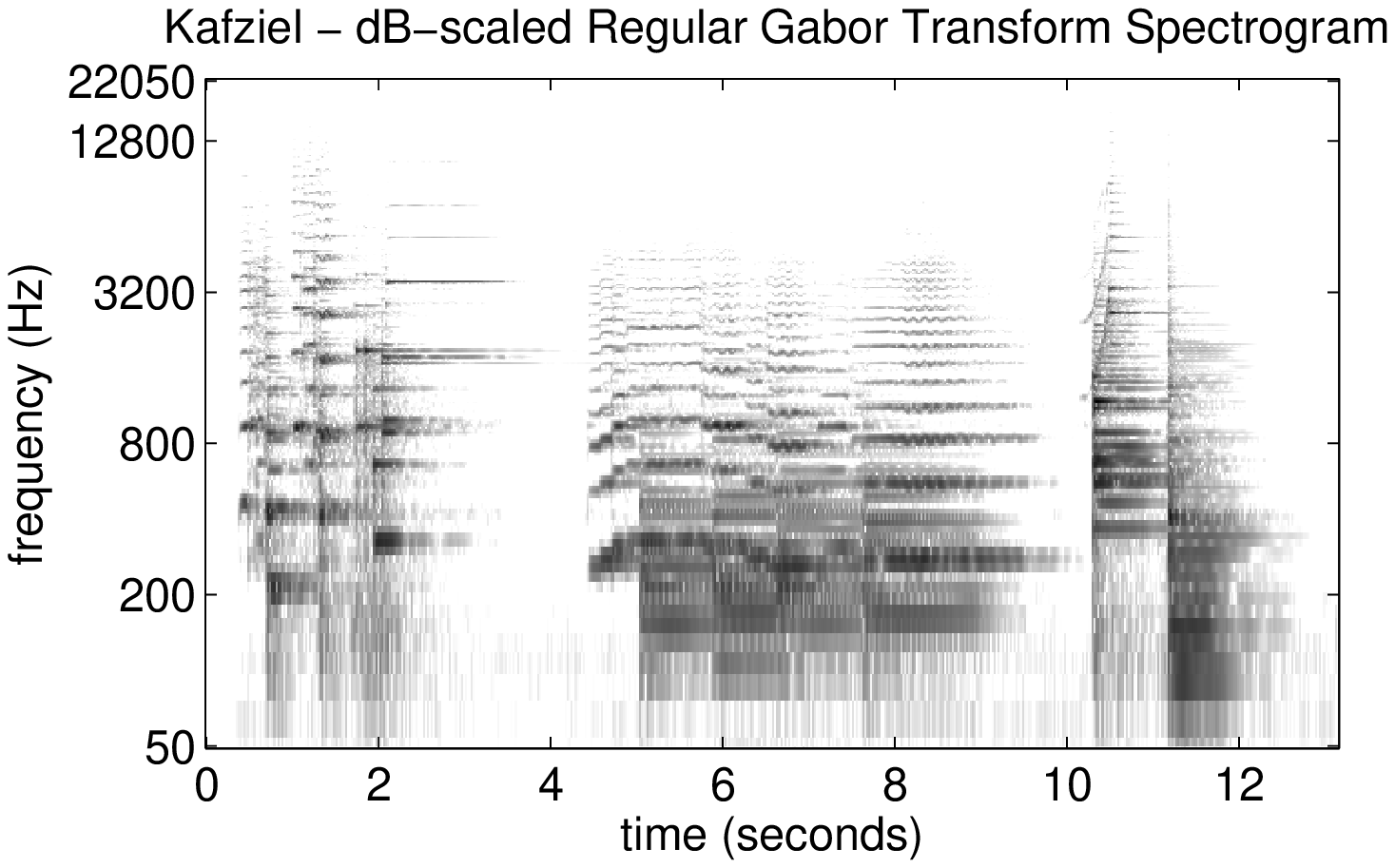}

\vspace{.5cm}
\includegraphics[width=8cm]{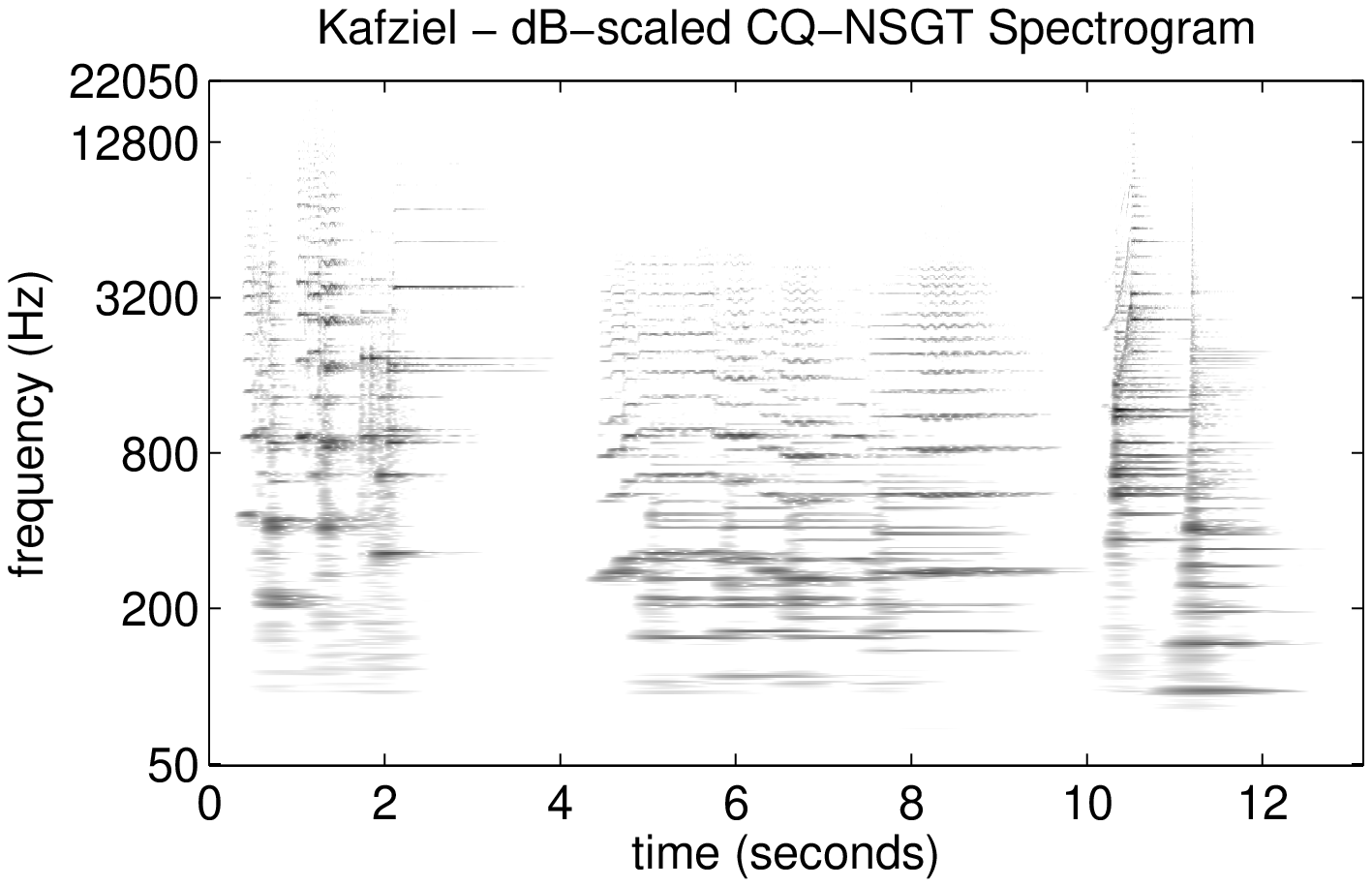}
\end{center}
\caption{Time-frequency representations on a logarithmically 
scaled frequency axis: STFT spectrogram (top) and constant-Q NSGT spectrogram (bottom).}
\label{fig:STFTvsNSGT}
\end{figure}

The main drawback of the CQ-NSGT is the inherent necessity to obtain a Fourier transform of  the entire signal 
prior to actual processing. This problem prohibits real-time implementation and is overcome by a slicing step, which preserves the perfect reconstruction property.
However, blocking effects and time-aliasing may be observed if the coefficients are modified in applications such as de-noising or transposition and time-shift of certain signal components.  While slicing the signal  naturally introduces a trade-off between delay and finest possible frequency resolution, the parameters can be chosen to suppress blocking artifacts and to leave the constant-Q coefficient structure intact.  

The rest of this paper is organized as follows. In Section~\ref{sec:NSGTfr} we introduce the concepts of frames as 
overcomplete, stable spanning sets, with a focus on nonstationary Gabor (NSG) systems and their properties. We recall the conditions
for these systems to constitute so-called \emph{painless} frames, a special case that allows for straightforward inversion.
Section~\ref{sec:CQNSGT} describes the construction of the CQ-NSGT by NSG frames with adaptivity in the frequency domain. This is the starting 
point for the sliCQ transform, which is explored in  Section~\ref{sec:SLICQ}. After giving the general idea, we describe 
interpretation of the sliCQ-coefficients in relation 
to the full-length transform in Section~\ref{sec:CQ_SLICQ}. Subsequently, Section~\ref{sec:NUMSIM} is 
concerned with an analysis of the transforms' numerical properties, in particular computation time and complexity, as well as the quality of 
approximation of the CQ-NSGT coefficients by the sliCQ, accompanied by a set of simulations. Finally, in Section~\ref{sec:EXPERI} the CQ-NSGT is applied 
and evaluated in the analysis and processing of real-life signals. The paper is closed by a short summary and conclusion.

%

\section{Nonstationary Gabor Frames}\label{sec:NSGTfr}

Frames,  first mentioned in~\cite{dusc52}, also cf.~\cite{chko07-1,ma09}, generalize (orthonormal) bases and allow for redundancy and thus  design  flexibility in  signal representations. Frames may be tailored to a specific  application or certain requirements such as a constant-Q frequency resolution. Loosely speaking, we wish to  represent a given signal of interest as a sum of the frame members $\varphi_{n,k}$,  weighted by coefficients $c_{n,k}$:
\begin{equation}\label{exp} f  =  \sum_{n,k} c_{n,k} \varphi_{n,k}. \end{equation} 
The double indexes $(n,k)$ allude to the fact that each atom has a certain location and concentration in time  and frequency.
Frame theory establishes  conditions under which an expansion of the form \eqref{exp} can be obtained with coefficients leading to stable, perfect reconstruction. 

For this contribution, we only consider frames for $\mathbb{C}^L$, that is vector spaces of finite, discrete signals, understood as functions $f,g$ on $\mathbb{C}^L$. We denote by $\langle f,g\rangle$ the inner product of $f$ and $g$, i.e.
 $\langle f,g\rangle = \sum_{l = 0}^{L-1}f[l]\overline{g[l]}$ and $\|f\|_2 = \sqrt{\langle f,f\rangle}.$
The structures introduced here can easily be extended to the Hilbert space of quadratically integrable functions, $\LtR$.

\subsection{Frames}\label{ssec:Frames}

Consider a collection of atoms $\varphi_{n,k}\in\mathbb{C}^L$ with  $(n,k) \in I_N \times I_K$ for finite index sets $I_N,I_K$. 
We then define the frame operator $\mathbf{S}$ by 
\begin{equation}\label{eq:frameop} 
\mathbf{S}f = \sum_{n,k} \langle f, \varphi_{n,k} \rangle \varphi_{n,k}, 
\end{equation} 
for all $f\in\mathbb{C}^L$. If the linear operator $\mathbf{S}$ is invertible on $\mathbb{C}^L$, then the set of functions $\{\varphi_{n,k}\}_{(n,k) \in I_N \times I_K}$, is a \emph{frame}\footnote{Note that, if  $\{\varphi_{n,k}, (n,k) \in I_N \times I_K\}$ is an orthonormal basis, then $\mathbf{S}$ is the identity operator.}. In this case, we may define a \emph{dual frame} by 
\begin{equation}\label{eq:candual}
\widetilde{\varphi_{n,k}} = \mathbf{S}^{-1}\varphi_{n,k} 
\end{equation}
and reconstruction from the coefficients $  c_{n,k} = \langle f, \varphi_{n,k} \rangle$ is straight-forward:
\begin{align*}\label{eq:recon}
f & = \mathbf{S}^{-1}  \mathbf{S} f =  \sum_{n,k} \langle f, \varphi_{n,k} \rangle \mathbf{S}^{-1}\varphi_{n,k}= \sum_{n,k} c_{n,k} \widetilde{\varphi_{n,k}}\\
  & = \mathbf{S}  \mathbf{S}^{-1} f =  \sum_{n,k} \langle f, \mathbf{S}^{-1}\varphi_{n,k} \rangle \varphi_{n,k}= \sum_{n,k} \langle f, \widetilde{\varphi_{n,k}} \rangle \varphi_{n,k}.
\end{align*}

We next introduce a  case of particular importance, the so-called {\it Gabor frames}, for which the elements $\varphi_{n,k}$ are obtained from a single window $\varphi$ by time- and frequency-shifts along a lattice. Let $\mathbf{T}_x$ and $\mathbf{M}_{\omega}$ denote a time-shift by $x$ and a frequency shift (or modulation) by $\omega$, i.e.
\begin{equation*}
  \mathbf{T}_x f[l] = f[l-x] \quad \text{and} \quad \mathbf{M}_\omega f[l] = e^{2\pi i l\cdot \omega/L}f[l],
\end{equation*}
where $l-x$ is considered modulo $L$. Furthermore, we use the normalization  
\[\mathcal{F}f[j] =\hat{f}[j] = \frac{1}{\sqrt{L}}\sum_{l=0}^{L-1} f[l] e^{-2\pi i l\cdot j/L}\] for  the discrete Fourier transform of $f$. It follows that $\mathcal{F}(\mathbf{T}_x f) = \mathbf{M}_{-x} \hat{f}$ and $\mathcal{F}(\mathbf{M}_{\omega} f) = \mathbf{T}_{\omega} \hat{f}$.

Fixing a time-shift parameter $a$ and a frequency-shift parameter $b$, with $L/a,L/b\in\mathbb{N}$, we call the collection of atoms $\mathcal{G} = \{\varphi_{n,k} = \mathbf{M}_{kb}\mathbf{T}_{na}\varphi\}_{(n,k)\in I_N \times I_K}$, with $I_N \times I_K = \ZZ_{L/a}\times\ZZ_{L/b}$, a \emph{Gabor system}. If $\mathcal{G}$ is a frame, it is called a Gabor frame. For Gabor frames, the frame coefficients are given by samples of the short-time Fourier transform (STFT) of $f$ with respect to the window $\varphi$:
\begin{align}\label{eq:STFT}
c_{n,k} = \langle f, \varphi_{n,k} \rangle &= \langle f, \mathbf{M}_{kb}\mathbf{T}_{na}\varphi \rangle\notag\\
&= \sum_{l=0}^{L-1}  f[l] \overline{\varphi [l-na]}e^{-2\pi i l\cdot kb/L}.
\end{align}
In a general setting, the inversion of the operator $\mathbf{S}$  poses a problem in numerical realization of frame analysis. However, for Gabor frames,  it was shown in~\cite{dagrme86}, that under certain conditions, usually fulfilled in practical applications,  $\mathbf{S}$ is diagonal, and a dual frame can be calculated easily. This situation of \emph{painless non-orthogonal expansions} can now be generalized to allow for adaptive resolution.

\subsection{Frequency-Adaptive Painless Nonstationary Gabor Frames}\label{ssec:Painless}
In classical Gabor frames, we  obtain all  samples of the STFT in \eqref{eq:STFT} by applying the same  window $\varphi$, shifted along a regular set of sampling points  and taking an FFT of the same length. In order to achieve adaptivity of the resolution in either time or frequency, we relax the regularity of classical Gabor frames to derive \emph{nonstationary Gabor frames}.

The original motivation for the introduction of NSGT was the desire to adapt both window size and sampling density in time, cf.~\cite{ja05-2,badohojave11}, in order to accurately resolve transient signal components. Here, we apply the same idea in frequency, i.e. adapt both the bandwidth and sampling density in frequency. From an algorithmic point of view, we apply a nonstationary Gabor system to the Fourier transform of the input signal. 

The windows are constructed directly in the frequency domain by taking real-valued filters $g_k$ centered at $\omega_k$.
The inverse Fourier transforms $\widecheck{g_k} := \mathcal{F}^{-1}g_k$ are the time-reverse impulse responses of the corresponding (frequency-adaptive) filters. 
Therefore, we let $\widecheck{g_k}$, $k\in I_K$, denote the members of  a finite collection of band-limited windows, well-localized in time, whose Fourier transforms $g_k = \mathcal{F}\widecheck{g_k}$ are centered around possibly irregularly (or, e.g. geometrically) spaced frequency points $\omega_k$.

Then, we select frequency dependent time-shift parameters (hop-sizes) $a_k$  as follows: if the \emph{support} (the interval where the vector is nonzero) of  $g_k$ is contained in an interval of length $L_k$, then  $a_k$ is chosen such that
\begin{equation}\label{eq:framecond1}
a_k\leq \frac{L}{L_k}\,\,\mbox{ for all } k. 
\end{equation}
In other words, the time-sampling points have to be chosen dense enough to guarantee \eqref{eq:framecond1}. If we denote by $g_{n,k}$ the modulation of $g_k$ by $-na_k$, i.e. $g_{n,k} = \mathbf{M}_{-n a_k}g_k$, then we obtain the  frame members $\varphi_{n,k}$ by setting
\begin{equation*}
    \varphi_{n,k} = \widecheck{g_{n,k}} = \mathcal{F}^{-1}(\mathbf{M}_{-n a_k}g_k) = \mathbf{T}_{n a_k}\widecheck{g_k}, 
\end{equation*}
where $k\in I_K$ and $n=0,\ldots,L/a_k-1$. The system $\mathcal{G}(\bd{g},\bd{a}):=\{g_{n,k} = T_{na_k}g_k\}_{n,k}$ is a \emph{painless nonstationary Gabor system}, as described in \cite{badohojave11}, for $\CC^L$. We also define $\bd{g}:= \{g_k\in\CC^L\}_{k\in I_K}$ and $\bd{a}:= \{a_k\}_{k\in I_K}$. By Parseval's formula, we see that the frame coefficients can be written as 
\begin{equation}\label{eq:CQ_coeff}
  c_{n,k} = \langle f,\widecheck{g_{n,k}} \rangle = \langle \hat{f},\mathbf{M}_{-n a_k}g_k \rangle.
\end{equation}
For convenience, we use the notation $c := \{c_k\}_{k\in I_K}:=\{\{c_{n,k}\}_{n=0}^{L/a_k-1}\}_{k\in I_K}$ to refer to the full set of coefficients and channel coefficients, respectively. By abuse of notation, we  indicate by $c\in \CC^{L/a_k \times |I_K|}$  that $c$ is an irregular array with $|I_K|$ columns, the $k$-th column possessing $L/a_k$ entries. The NSG coefficients can be computed using the following algorithm.

\begin{algorithm}
\caption{NSG analysis: $c = \textbf{CQ-NSGT}_L(f,\bd{g},\bd{a})$}\label{alg:nsganalysis}
\begin{algorithmic}[1]
\State \textbf{Initialize} $f,g_k$ for all $k\in I_K$
\State $f\gets \textbf{FFT}_L(f)$
\For{$k\in I_K,~n=0,\ldots,L/a_k-1$}
\State $c_k \gets \sqrt{L/a_k}\cdot\textbf{IFFT}_{L/a_k}(f\overline{g_k})$
\EndFor
\end{algorithmic}
\end{algorithm}

Here $\textbf{(I)FFT}_N$ denotes a (inverse) Fast Fourier transform of length $N$, including the necessary periodization or zero-padding preprocessing to convert the input vector to the correct length $N$. The analysis algorithm above is complemented by Algorithm \ref{alg:nsgsynthesis}, an equally simple synthesis algorithm that synthesizes a signal $\tilde{f}$ from a set of coefficients $c$.

\begin{algorithm}
\caption{NSG synthesis: $\tilde{f} = \textbf{iCQ-NSGT}_L(c,\tilde{\bd{g}},\bd{a})$}\label{alg:nsgsynthesis}
\begin{algorithmic}[1]
\State \textbf{Initialize} $c_{n,k},\widetilde{g_k}$ for all $n=0,\ldots,L/a_k-1$, $k\in I_K$
\For{$k\in I_k$}
\State $f_k \gets \sqrt{a_k/L} \cdot \textbf{FFT}_{L/a_k}(c_k)$
\EndFor
\State $\tilde{f}\gets \sum_{k\in I_K} f_k\widetilde{g_k}$
\State $\tilde{f}\gets \textbf{IFFT}_L(\tilde{f})$
\end{algorithmic}
\end{algorithm}

If $\mathcal{G}(\bd{g},\bd{a})$ and $\mathcal{G}(\tilde{\bd{g}},\bd{a})$ are a pair of dual frames, then we can reconstruct a function perfectly from its NSG analysis coefficients. For more details and a proof of the following propositions, see Appendix \ref{sec:CQprop}.

\begin{Pro}\label{pro:reconst}
  Let $\mathcal{G}(\bd{g},\bd{a})=\{g_{n,k} = T_{na_k}g_k\}_{n,k}$ and $\mathcal{G}(\tilde{\bd{g}},\bd{a})=\{\widetilde{g_{n,k}} = T_{na_k}\widetilde{g_k}\}_{n,k}$ be a pair of dual frames. If $c$ is the output of 
  $\textbf{\emph{CQ-NSGT}}_L(f,\bd{g},\bd{a})$ (Algorithm \ref{alg:nsganalysis}), then the output $\tilde{f}$ of $\textbf{\emph{iCQ-NSGT}}_L(c,\tilde{\bd{g}},\bd{a})$ (Algorithm \ref{alg:nsgsynthesis}) equals $f$, i.e.
  \begin{equation}
    \tilde{f} = f, \quad \text{for all } f\in\CC^L.
  \end{equation}
\end{Pro}

The remaining problem is to ascertain that $\mathcal{G}(\bd{g},\bd{a})$ is a frame and to compute the dual frame. The following proposition is a discrete version of an equivalent result for NSG systems in $\LtR$ and achieves both, using the painless case condition \eqref{eq:framecond1}.

\begin{Pro}\label{pro:NSGpainless}
  Let $\mathcal{G}(\bd{g},\bd{a})$ an NSG system satisfying \eqref{eq:framecond1}. This system is a frame if and only if 
  \begin{equation}\label{eq:painlessframe1}
    0 < \sum_{k\in I_K} \frac{L}{a_k} |g_k[j]|^2 < \infty,\quad \text{ for all } j = 0,\ldots,L-1
  \end{equation}
  and the generators of the canonical dual frame $\mathcal{G}(\tilde{\bd{g}},\bd{a})$ are given by
  \begin{equation}\label{eq:painlessframe2}
    \widetilde{g_k}[j] = \frac{g_k[j]}{\sum_{l\in I_K} \frac{L}{a_l} |g_l[j]|^2}.
  \end{equation}
\end{Pro}

In the next section, we construct a constant-Q NSG system satisfying \eqref{eq:framecond1} and \eqref{eq:painlessframe1}.

\begin{Rem}
Note that 
NSG frames can be equivalently used to design general nonuniform filter banks~\cite{LiNg97,NaBaSm93} in a similar manner. 
\end{Rem}

\section{The CQ-NSGT Parameters: Windows and Lattices}\label{sec:CQNSGT}

The parameters of the NSGT can be designed as to implement various frequency-adaptive transforms. Here, we focus on the parameters leading to an NSGT with constant-Q frequency resolution, suitable for the analysis and processing of music signals, as discussed in the introduction.
In constant-Q analysis, the functions $g_k$ are considered to be filters with support of length $L_k \leq L$ centered at frequency $\omega_k$ (in samples), such that for the bins corresponding to a certain frequency range, the respective center frequencies and lengths have (approximately) the same ratio. Using these filters, the CQ-NSGT coefficients $c_{n,k}$ are obtained via Algorithm \ref{alg:nsganalysis}, where $k$ indexes the frequency bins, and $n = 0,\ldots,L/a_k-1$.

\begin{table}[t!]
\caption{Center frequency and bandwidth values}\label{tab:freq_band}\vspace{-10pt}
 \renewcommand*{\arraystretch}{1.3}
 \begin{normalsize}
\begin{center}
  \begin{tabular}{|c|c|c|}\hline         
    $k$ & $\xi_k$ & $\Omega_k$ \\\hline\hline
    $0$ & $0$ & $2\xi_{\text{min}}$ \\\hline       
    $1,\ldots,K$ & $\xi_{\text{min}}2^{\frac{k-1}{B}}$ & $\xi_k/Q$\\\hline       
    $K+1$ & $\xi_s/2$ & $\xi_s-2\xi_K$\\\hline       
    $K+2,\ldots,2K+1$ & $\xi_s-\xi_{2K+2-k}$ & $\xi_{2K+2-k}/Q$\\\hline\end{tabular}
  \end{center}
  \end{normalsize}
  \vspace{-15pt}
\end{table}

As detailed in \cite{dogrhove11}, the construction of the filters for the CQ-NSGT depends on the following parameters: minimum and maximum frequencies $\xi_{\text{min}}$ and $\xi_{\text{max}}$ (in Hz), respectively, the sampling rate $\xi_s$, and the number of bins per octave $B$. The center frequencies $\xi_k$ satisfy $\xi_k = \xi_{\text{min}}2^{\frac{k-1}{B}}$, similar to the classical CQT in \cite{BR91}, for $k = 1,\ldots,K$, where $K$ is an integer such that $\xi_{\text{max}}\leq \xi_K < \xi_s/2$, the Nyquist frequency. Note that the correspondence between $\xi_k$ and $\omega_k$ is the conversion ratio from Hz to samples, as detailed in the next paragraphs.

The bandwidths are set to be $\Omega_k = \xi_{k+1}-\xi_{k-1}$, for $k = 2,\ldots,K-1$, which lead to a constant Q-factor $Q=\xi_k/\Omega_k = (2^{\frac{1}{B}}-2^{-\frac{1}{B}})^{-1}$, while $\Omega_1$ and $\Omega_K$ are taken to be $\xi_1/Q$ and $\xi_K/Q$, respectively. Since the signals are real-valued, additional filters are considered which are positioned in a symmetric manner with respect to the Nyquist frequency. Moreover, to ensure that the union of filter supports cover the entire frequency axis, filters with center frequencies corresponding to the zero frequency and the Nyquist frequency are included. The values for $\xi_k$ and $\Omega_k$ over all frequency bins are summarized in Table \ref{tab:freq_band}.

%
With these center frequencies and bandwidths, the filters $g_k$ are set to be $g_k[j] = H((j\xi_s/L-\xi_k)/\Omega_k)$, for $k = 1,\ldots,K,K+2,\ldots,2K+1$, where $H$ is some continuous function centered at $0$, positive inside and zero outside of $]-1/2,1/2[$, i.e. each $g_k$ is a sampled version of a translated and dilated $H$. Meanwhile, $g_0$ and $g_{K+1}$ are taken to be plateau functions centered at the zero and the Nyquist frequencies respectively. Thus, each filter $g_k$ is centered at $\omega_k = \xi_kL/\xi_s$ and has support $L_k = \Omega_k L/\xi_s$. 

It is easy to see that this choice of $\mathcal{G}(\bd{g},\bd{a})$ satisfies the conditions of Proposition~\ref{pro:NSGpainless} 
 for any sequence $\bd{a}$ with $L/a_k \geq L_k$ for all $k\in I_K = \{0,\ldots,2K+1\}$. Note that while $a_k$ might be rational, $L/a_k$ must be integer-valued. Consequently, perfect reconstruction of the signal is obtained from the coefficients $c_{n,k}$ by applying Algorithm \ref{alg:nsgsynthesis} with a dual frame, e.g. the canonical dual given by \eqref{eq:painlessframe2}. 
%
%

\section{Real-time processing and the sliCQ}\label{sec:SLICQ}

The CQ-NSGT implementation introduced in the previous sections a priori relies  on a Fourier  transform  of the entire signal. This contradicts the idea of real-time applications, which require bounded delay in  processing  incoming samples 
and linear over-all complexity.  These requirements can  be satisfied by applying the CQ-NSGT in a blockwise manner, i.e. to (fixed length) slices of the input signal. However, the slicing process involves two important challenges: 
First, the windows $h_m$ used for cutting the signal must be smooth and zero-padding has to be applied to suppress time-aliasing and blocking artifacts when coefficient-modification occurs. Second, the coefficients issued from the block-wise transform should be equivalent to the CQ-coefficients obtained from a full-length CQ-NSGT. This can be achieved to high precision by careful choice of both the slicing windows $h_m$ and the analysis windows $g_k$ used in the CQ-NSGT.
\subsection{Structure of the sliCQ transform}\label{Se:Struc}

 \begin{figure}[t!]
     \centering
       \includegraphics[width=8.5cm]{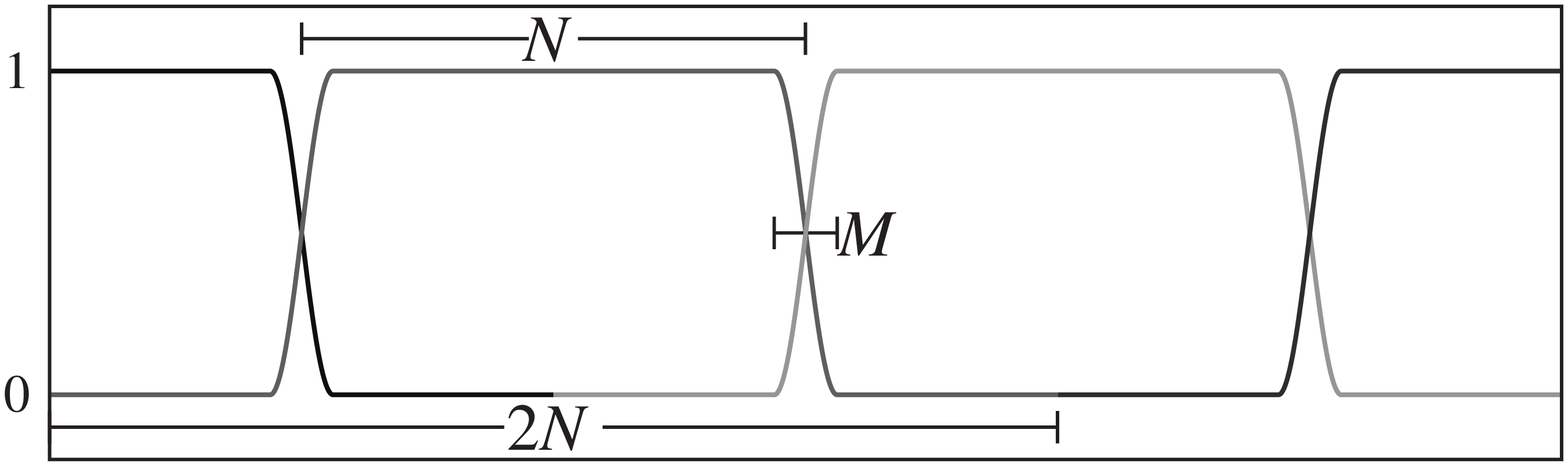}
     \caption{Tukey windows used in the slicing process. Note that the chosen amount of zero-padding leads to a half-overlap situation.}\label{fig:tukeywins}
  \end{figure}
  
We now summarize the individual steps of the sliCQ algorithm and introduce the involved parameters.
\begin{itemize}
  \item[I)] Sliced constant-Q NSGT analysis:
    \begin{enumerate}
      \item Cut the signal $f\in\CC^L$ into overlapping slices $f_m$ of length $2N$ by multiplication with uniform translates of a slicing window $h_0$, centered at $0$.
      \item For each $f^m$, obtain coefficients  $c^m \in\CC^{2N/a_k \times |I_K|}$, by applying  $\textbf{CQ-NSGT}_{2N}(f,\bd{g},\bd{a})$ (Algorithm~\ref{alg:nsganalysis}).
      \item  Due to the overlap of the slicing windows, cf.~Figure~\ref{fig:tukeywins}, each time index 
      is related to two consecutive slices. 
      For visualization and processing, the slice coefficients $c^m$ are re-arranged into a $2$-layer array $s$, with $s := \{s^l\}_{l\in\{0,1\}} \in \mathbb{C}^{2\times L/a_k\times |I_K|}$, cf. Figure~\ref{fig:slicqarray}.
    \end{enumerate}
  \item[II)] Sliced constant-Q NSGT synthesis:
    \begin{enumerate}
      \item Retrieve $c^{m}$ by partitioning $s$.
      \item Compute the dual frame $\mathcal{G}(\tilde{\bd{g}},\bd{a})$ for $\mathcal{G}(\bd{g},\bd{a})$ and, for all $m$, $\tilde{f}^m = \textbf{iCQ-NSGT}_{2N}(c^m,\tilde{\bd{g}},\bd{a})$ (Algorithm~\ref{alg:nsgsynthesis}).
      \item Recover $f$ by (windowed) overlap-add.
    \end{enumerate}
\end{itemize}
Note that $L$ must be a multiple of $2N$; this is achieved by zero-padding, if necessary. 
By construction, the positions $(n,k)$ of the coefficients in $s^l$ reflect their time-frequency position with respect to the full-length signal, for $l = 0,1$. 

\begin{figure}[t!]
    \centering
      \includegraphics[width=8.5cm]{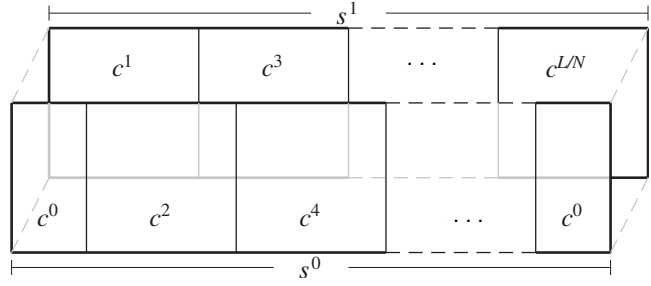}
    \caption{Structure of the sliCQ coefficients - schematic illustration}\label{fig:slicqarray}
  \end{figure}
\subsection{Computation of a sliced constant-Q NSGT}

The \emph{sliced constant-Q NSGT} (sliCQ) coefficients of $f$  with respect to $h_0$ and $\mathcal{G}(\bd{g},\bd{a})$ and slice length $2N$ are obtained according to the following algorithm. 
\begin{algorithm}
\caption{sliCQ analysis: $s = \textbf{sliCQ}_{L,N} (f,h_0,\bd{g},\bd{a})$}\label{alg:slicqana}
\begin{algorithmic}[1]
\State \textbf{Initialize} $f,h_0,g_k$ for all $k\in I_K$
\State $m \gets 0$
\For{$m=0,\ldots L/N-1$}
  \For{$j=0,\ldots 2N-1$}
    \State $f^m[j] \gets f\bd{T}_{mN}h_0[j+(m-1)N]$
  \EndFor
  \State $c^m \gets \textbf{CQ-NSGT}_{2N}(f,\bd{g},\bd{a})$
  \State $l \gets (m\mod 2)$
  \For{$k\in I_K,~n^s = 0,\ldots,2N/a_k-1$}
    \State $s^l_{n^s+(m-1)N/a_k,k} \gets c^m_{n^s,k}$
  \EndFor
\EndFor
\end{algorithmic}
\end{algorithm}

Note that in this and the following algorithm, negative indices are used in a circular sense, with respect to the maximum admissible index, e.g. $f[-j]:= f[L-j]$ or $s^l_{-n,k}:= s^l_{L/a_k-n,k}$. As the CQ-NSGT analysis before, Algorithm~\ref{alg:slicqana} is complemented by a synthesis algorithm with similar structure, Algorithm~\ref{alg:slicqsyn}, that synthesizes a signal $\tilde{f}$ from a $2$-layer coefficient array $s$. 

\begin{algorithm}
\caption{sliCQ synthesis: $\tilde{f} = \textbf{isliCQ}_{L,N} (s,\tilde{h}_0,\tilde{\bd{g}},\bd{a})$}\label{alg:slicqsyn}
\begin{algorithmic}[1]
\State \textbf{Initialize} $s,\tilde{h}_0,\tilde{g}_k$ for all $k\in I_K$
\State $m \gets 0$
\State $\tilde{f} \gets \bd{0}_L$
\For{$m=0,\ldots L/N-1$}
  \State $l \gets (m\mod 2)$
  \For{$k\in I_K,~n^s = 0,\ldots,2N/a_k-1$}
    \State $c^m_{n^s,k} \gets s^l_{n^s+(m-1)N/a_k,k}$
  \EndFor
  \State $\tilde{f}^m \gets \textbf{iCQ-NSGT}_{2N}(c^m,\tilde{\bd{g}},\bd{a})$
  \For{$j=0,\ldots 2N-1$}
    \State $\tilde{f}[j+(m-1)N] \gets$
    \Statex \hspace{40pt} $\tilde{f}[j+(m-1)N] + \tilde{f}^m[j]\tilde{h}_0[j-N]$
  \EndFor
\EndFor
\end{algorithmic}
\end{algorithm}

The following proposition states that $f$ is perfectly recovered from its sliCQ coefficients by applying Algorithm \ref{alg:slicqsyn}, see Appendix \ref{sec:sliCQprop} for a proof.

\begin{Pro}\label{Prop:sliPerfRec}
  Let $\mathcal{G}(\bd{g},\bd{a})$ and $\mathcal{G}(\tilde{\bd{g}},\bd{a})$ be dual NSG systems for $\CC^{2N}$. Further let $h_0,\tilde{h}_0 \in \CC^L$ satisfy 
  \begin{equation}\label{dualwincond}
   \sum_{m=0}^{L/N-1} \bd{T}_{mN} \left(h_0\overline{\tilde{h}_0}\right) \equiv 1.
  \end{equation}
 If $s$ is the output of $\textbf{\emph{sliCQ}}_{L,N} (f,h_0,\bd{g},\bd{a})$ (Algorithm \ref{alg:slicqana}), then the output $\tilde{f}$ of $\textbf{\emph{isliCQ}}_{L,N} (s,\tilde{h}_0,\tilde{\bd{g}},\bd{a})$ (Algorithm \ref{alg:slicqsyn}) equals $f$,~i.e., 
    $\tilde{f} = f$.
\end{Pro}
\subsection{The relation between CQ-NSGT and sliCQ}\label{sec:CQ_SLICQ}
To maintain perfect reconstruction in the final overlap-add step in 
Algorithm~\ref{alg:slicqsyn}, we assume 
 \begin{equation}\label{eq:BUPU}
    h_m = \bd{T}_{mN} h_0 \text{ with } \sum_{m=0}^{L/N-1} h_m\equiv 1,
  \end{equation}
  and use a dual window $\tilde{h}_0$ satisfying \eqref{dualwincond} in the synthesis process.
  
  Another obvious option for the design of the slicing windows is to require
$\sum_m h_m^2\equiv 1$, which would allow for using the same windows in the final overlap-add step. However, if we want to approximate the true CQ-coefficients as obtained from a full-length transform, \eqref{eq:BUPU} is the more favorable condition.

 In our implementation, \emph{slicing} of the signal is accomplished 
  by a uniform partition of unity constructed from a Tukey window $h_0$ with essential length $N$ and transition areas
  of length $M$, for some $N,M\in\NN$ with $M < N$ (usually $M\ll N$). The slicing windows are symmetrically zero-padded to length $2N$, reducing time-aliasing significantly.
  The uniform partition condition \eqref{eq:BUPU} leads to close approximation of the full-length CQ-NSGT by sliCQ. This correspondence between the sliCQ and the corresponding full-length CQ-NSGT is made explicit in the following proposition, proven in Appendix~\ref{sec:sliCQprop}.
\begin{Pro}\label{pro:CQ_sliCQ}
  Let $\mathcal{G}(\bd{g}^\mathcal{L},\bd{a})$ be a nonstationary Gabor system for $\CC^L$. 
   Further, let  $h_0\in\CC^L$ be such that \eqref{eq:BUPU} holds
  and define $g_k\in\CC^{2N}$, for all $k\in I_K$ by
  \begin{equation*}
    g_k[j] = g^\mathcal{L}_k[jL/(2N)].
  \end{equation*}
For $f\in\CC^L$, denote by $c\in\mathbb{C}^{L/a_k\times |I_K|}$ the CQ-NSGT coefficients of $f$ with respect to $\mathcal{G}(\bd{g}^\mathcal{L},\bd{a})$
  and by $s\in\mathbb{C}^{2\times L/a_k\times |I_K|}$ the sliCQ coefficients of $f$ with respect to $h_0$ and $\mathcal{G}(\bd{g},\bd{a})$. Then 
  \begin{align}
    \lefteqn{|s^0_{n,k}+s^1_{n,k}-c_{n,k}|}\nonumber\\
      & \leq \|f\|_2 \Big(\|(1-h_{0}-h_{1})\bd{T}_{n^s a_k}\widecheck{g^\mathcal{L}_k}\|_2 \nonumber \\
      & + \|(h_0+h_{1})\sum_{j= 1}^{\frac{L}{2N}-1} \bd{T}_{n^s a_k+2jN} \widecheck{g^\mathcal{L}_k}\|_2\Big) \label{eq:inner_outer}
  \end{align}
  for $n = mN/a_k +n^s$, with $m = 0,\ldots, L/N-1$ and $n^s = 0,\ldots, N/a_k-1$. 
\end{Pro}
\begin{Rem}
 In practice, $\widecheck{g^\mathcal{L}_k}$ is chosen such that the translates $\bd{T}_{na_k} \widecheck{g^\mathcal{L}_k}$ are \emph{essentially concentrated} in 
  \begin{equation*}
    I_{N,M} = [-\frac{N-M}{2},N+\frac{N-M}{2}],
  \end{equation*}
 i.e. $\|\bd{T}_{na_k} \widecheck{g^\mathcal{L}_k}\chi_{\RR\setminus I_{N,M}}\|_2 \ll \|\bd{T}_{na_k} \widecheck{g^\mathcal{L}_k}\|_2$, for all $n = 0,\ldots, N/a_k-1$. Therefore, the value of \eqref{eq:inner_outer} is negligibly small. While more precise 
 estimates of the error are beyond the scope of the present contribution, numerical  evaluation of the approximation quality is given in Section~\ref{ssec:approx}.
\end{Rem}
  
  As a consequence of the previous proposition, we define the \emph{sliCQ spectrogram} as $|s^0+s^1|^2$ and propose to simultaneously treat $s^0_{n,k}$ and $s^1_{n,k}$, corresponding to the same time-frequency position, when processing the coefficients.
  
\section{Numerical Analysis and Simulations}\label{sec:NUMSIM}

In this section we treat the computational complexity 
of  CQ-NSGT and
sliCQ  and how they compare to one another. In \cite{dogrhove11} it was shown that despite superlinear complexity,
	 CQ-NSGT outperforms state-of-the-art implementations of the classical constant-Q
transform. Since sliCQ is a
linear cost algorithm, it  further improves 
the efficiency of the CQ-NSGT for sufficiently long signals.
 Section~\ref{ssec:approx}  provides experimental results
confirming  the good approximation of CQ-NSGT  by the corresponding sliCQ 
coefficients, cf.~Proposition~\ref{pro:CQ_sliCQ}.

The CQ-NSGT and sliCQ Toolbox (for MATLAB and Python) used in this contribution
is available  at \url{http://www.univie.ac.at/nonstatgab/slicq}, 
alongside extended experimental results complementing those presented in Section 
\ref{sec:EXPERI}. 

\subsection{Computation Time and Computational Complexity}
 We assume  the number of filters $|I_K|$ in the CQ-NSGT to be independent of the signal length $L$ and Proposition \ref{pro:NSGpainless} to hold, in particular $L/a_k \geq L_k$.  The support size $L_k$ of each filter $g_k$ depends on $L$. Hence, the number of operations for Algorithm \ref{alg:nsganalysis} is as follows:
\begin{equation*}
  \OO \Big(\underbrace{L~\log\left(L\right)}_{\textbf{FFT}_L} +\sum \limits_{k\in I_K} \underbrace{L/a_k~\log\left(L/a_k\right)}_{\textbf{IFFT}_{L/a_k}}+\underbrace{L_k}_{f\cdot\overline{g_k}}\Big).
\end{equation*}
With $L_k$ and $L/a_k$ bounded by $L$, this can be simplified to $\OO (L~\log L)$. 

The computation of the dual frame involves inversion of the multiplication operator $\mathbf{S}$ and applying the resulting operator $\mathbf{S}^{-1}$ to each filter. This results in $\OO(2\sum_{k\in I_K} L_k) =  \OO(L)$ operations, where the support of the $g_k$ was taken into account.

Complexity of Algorithm \ref{alg:nsgsynthesis} can be derived to be $\OO (L~\log L)$, analogous to Algorithm \ref{alg:nsganalysis}.\\

For $\textbf{sliCQ}_{L,N}$ (Algorithm \ref{alg:slicqana}), we assume the slice length $2N$ to be independent of $L$, resulting in a computational complexity of 
\begin{equation*}
  \OO \Big(\underbrace{L/N}_{\#\text{slices}}\cdot\big(\underbrace{2N ~\log\left(2N\right)}_{\textbf{CQ-NSGT}_{2N}} +\underbrace{2N}_{f\cdot\bd{T}_{mN}\overline{h_0}}\big)\Big) = \OO (L).
\end{equation*}
Both the dual frame and $\tilde{h}_0$ can be precomputed independent of $L$, whilst Algorithm \ref{alg:slicqsyn} is of complexity $\OO (L)$, analogous to Algorithm \ref{alg:slicqana}.

\subsection{Performance evaluation}

\begin{figure}[t!]
  \begin{center}
  \includegraphics[width=8.5cm]{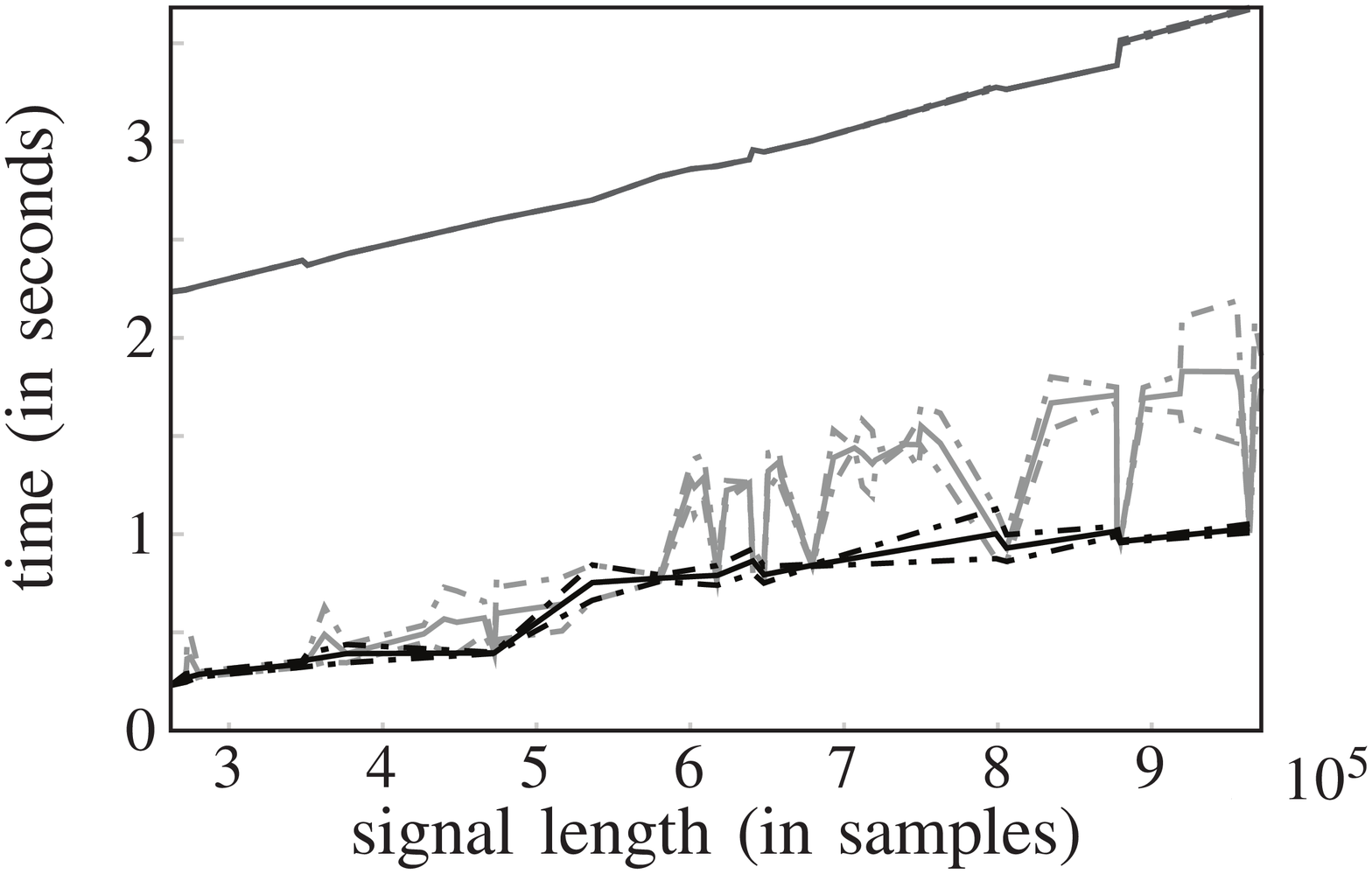}
  \vspace{4pt}
  \end{center}
  \caption{Computation time versus signal length of the CQ transform (dark gray) and CQ-NSGT. For the CQ-NSGT we show separate
           graphs including (light gray), respectively neglecting prime signal lengths (black). Graphs show the mean performance (solid) and variance (dashed)
           over 50 iterations.}\label{fig:CQvsNSGT}
\end{figure}
\begin{figure}[t!]
  \begin{center}
  \includegraphics[width=8.5cm]{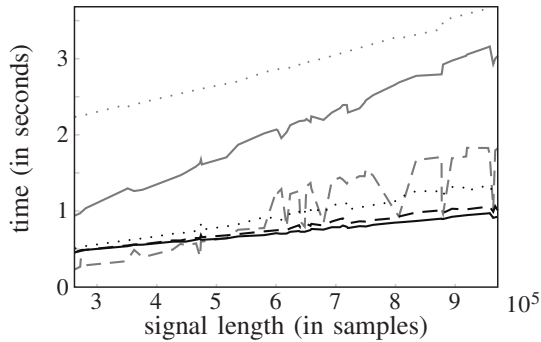}
  \vspace{4pt}
  \end{center}
  \caption{Computation time versus signal length of the CQ transform (dotted gray), CQ-NSGT (dashed gray) and various sliCQ transforms.
          The sliCQ transforms were taken with slice lengths 4096 (solid gray), 16384 (dotted black), 32768 (dashed black) and 65536 (solid black)
          samples.}\label{fig:NSGTvsSLICQ}
\end{figure}

A comparison  of the CQ-NSGT algorithm with previous constant-Q implementations was given  in \cite{dogrhove11}. Figure~\ref{fig:CQvsNSGT} reproduces and extends some of the results; it shows, for both the 
constant-Q implementation provided in \cite{klsc10} and  CQ-NSGT, mean computation duration and variance for analysis followed by reconstruction, against signal length. 
The plot also illustrates the dependence of CQ-NSGT on the prime factor decomposition of the signal length $L$.

Figure \ref{fig:NSGTvsSLICQ} illustrates the performance of sliCQ  compared to the constant-Q and CQ-NSGT algorithms shown in Figure~\ref{fig:CQvsNSGT}. Linearity 
of the sliCQ algorithm becomes obvious, deviations occurring due to unfavorable FFT lengths $2N/a_k$ in $\textbf{(i)CQ-NSGT}_{2N}$. Performance improvements for increasing slice length can be attributed to the advanced nature of MATLAB's internal FFT algorithm, as compared to the current implementation of the sliCQ framework.

The performance of the involved algorithms does not depend on signal content. Consequently, random signals were used in the performance experiments, although we implicitly assumed the signals
to be sampled at $44.1$~kHz. All the results represent transforms with $48$ bins per octave, minimum frequency $50$~Hz and maximum frequency $22$~kHz, in Section \ref{sec:EXPERI} a maximum frequency of $20$~kHz is used instead. For a more comprehensive comparison 
of the CQ-NSGT to previous constant-Q transforms, please refer to \cite{dogrhove11}. Results for other parameter values do not differ drastically and are omitted.

All computation time experiments were run in MATLAB R2011a on a 3 Gigahertz Intel Core 2 Duo machine with 2 Gigabytes of RAM running Kubuntu 10.04 using the MATLAB toolboxes available at \url{http://www.elec.qmul.ac.uk/people/anssik/cqt/} and \url{http://www.univie.ac.at/nonstatgab/}.

\subsection{Approximation properties}\label{ssec:approx}

\begin{figure}[t!]
  \begin{center}
  \includegraphics[width=6cm]{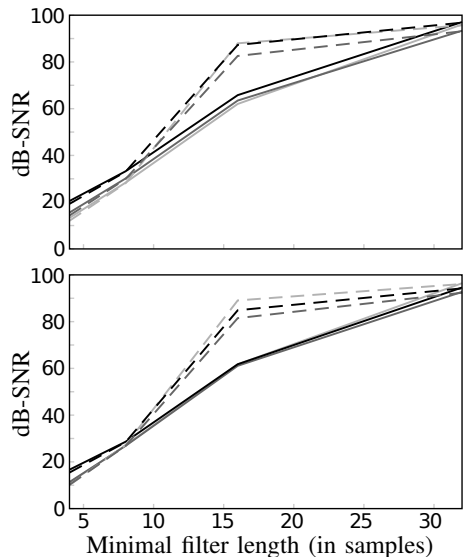}
  \end{center}
  \caption{SliCQ coefficient approximation error against the minimal admissible bandwidth for Set $1$ (top) and Set $2$ (bottom). All transforms use 
           Blackman-Harris windows in the CQ-NSGT step. Solid and dashed lines represent long ($1/4$ slice length) and short ($1/128$ slice length) transition areas respectively, while 
           colors correspond to the slice length: $4096$ (light gray), $16384$ (dark gray) and $65536$ samples (black).}
  \label{fig:CoeffQual}
\end{figure}

To verify the approximate equivalence of the sliCQ coefficients to those of a
full-length CQ-NSGT and thus to a constant-Q transform,  we computed the norm difference between $s^0+s^1$ and $c$ as in Proposition \ref{pro:CQ_sliCQ}, for two sets of fundamentally different signals. Set $1$ contains $50$ random, complex-valued signals of $2^{20}$ samples length,
while Set $2$ consists of $90$ music samples of the same length, sampled at $44.1$~kHz each, covering pop, rock, jazz and classical genres. The signals of the second set
are well-structured and often well-concentrated in the time-frequency plane, characteristics that the first set lacks completely. 

For discretization reasons as well as to achieve good concentration of $\widecheck{g^\mathcal{L}_k}$ in Proposition \ref{pro:CQ_sliCQ}, sliCQ implementations must impose a lower bound on the length of $g_k$.
Approximation results for various lower bounds on the filter length are summarized in Figure~\ref{fig:CoeffQual}, showing the mean approximation quality over the whole set.


All errors are given in signal-to-noise ratio, scaled in dB: 
\begin{equation*}
 20 \log_{10} \frac{\|c\|_2}{\|c-(s^0+s^1)\|_2}
\end{equation*}

Figure~\ref{fig:CoeffQual} shows that, independent of other parameters, a minimal filter length smaller than
$8$ samples leads to a representation that is visibly different from, while values above $16$ samples yield coefficients that are largely equivalent to those of
a constant-Q transform. We can see that the slice length itself has rather small influence on the results, while the interplay of slicing window shape, specified by
the ratio of transition area length to slice length, and minimal filter length is illustrated nicely; remarkably, this ratio influences the approximation quality mainly for moderately well localized filters.
This is in correspondence with the characterization given in \eqref{eq:inner_outer}: the \emph{circular overspill}, given by the second term of the right hand side in \eqref{eq:inner_outer}, depends on the shape and support of the sum of two adjacent slicing windows, in particular for moderately well localized filters. If the windows are very well localized, the overspill 
is   small independent of the particular shape of the slicing area. On the other hand, very badly localized windows make the distinct influence of the slicing windows negligible.  
Finally, a comparison of the top and bottom graphs in Figure~\ref{fig:CoeffQual} shows that the approximation quality is largely independent of the signal class. For Set $1$ the variance is generally negligible ($<0.1$~dB) and was omitted. Despite some outliers in Set $2$, we have found the approximation quality to depend on the minimal filter length in a stable way, cf. Figure~\ref{fig:CoeffQual3}. These outliers can be attributed to signals particularly sparse (smaller error) or dense (larger error) in low frequency regions, where $\widecheck{g_k^\mathcal{L}}$ is least concentrated.


 \begin{figure}[t!]
  \begin{center}
  \includegraphics[width=8.5cm]{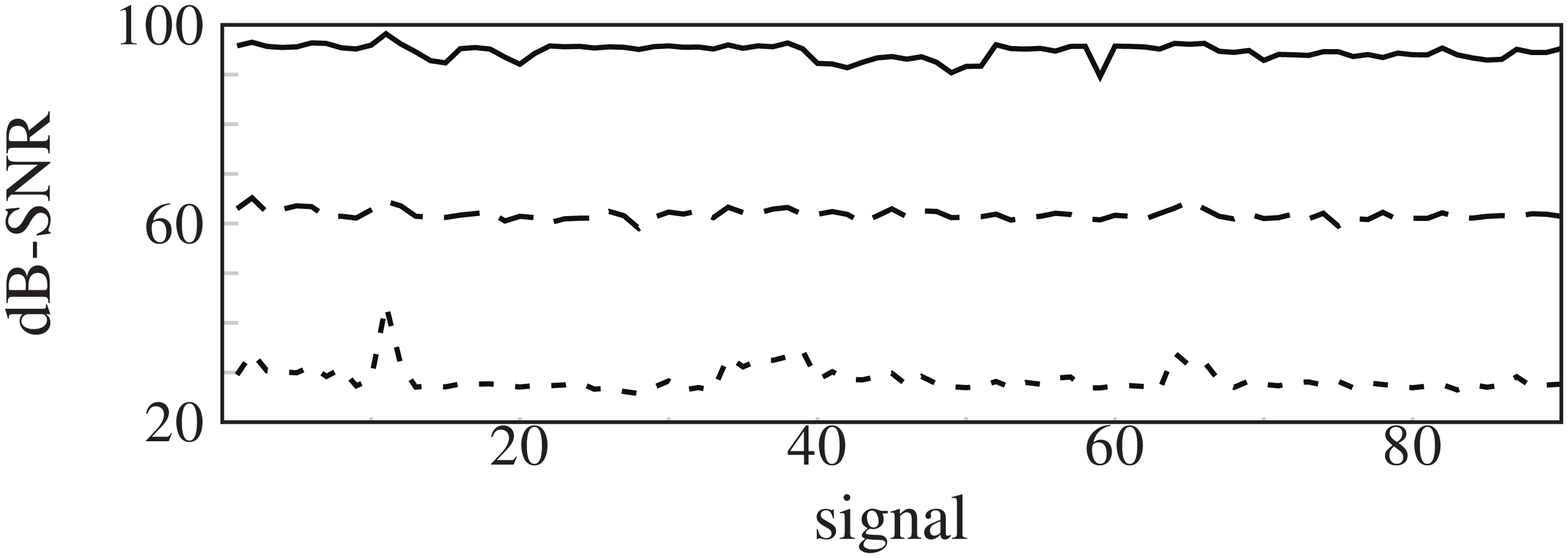}
  \end{center}
  \caption{Coefficient approximation error \eqref{eq:inner_outer} for all signals from Set $2$ and slice and transition length of $65536$, resp. $16384$ samples. Line style indicates 
           the minimal filter length: $8$ (dotted), $16$ (dashed) and $32$ (solid) samples.}
  \label{fig:CoeffQual3}
\end{figure}

\section{Experiments on Applications}\label{sec:EXPERI}\texttt{}

Experiments in \cite{dogrhove11} show how the CQ-NSGT can be applied in the processing of signals taking advantage of the logarithmic frequency scaling and the perfect reconstruction property. In particular, the transposition of a harmonic structure amounted to just a translation of the spectrum along frequency bins, while the masking of the CQ-NSGT coefficients allowed for the extraction or suppression of a component of the signal. In our experiment, we show that the two procedures can be used to modify a portion of a signal.

\begin{figure}[t!]
\begin{center}
\includegraphics[width=7.5cm]{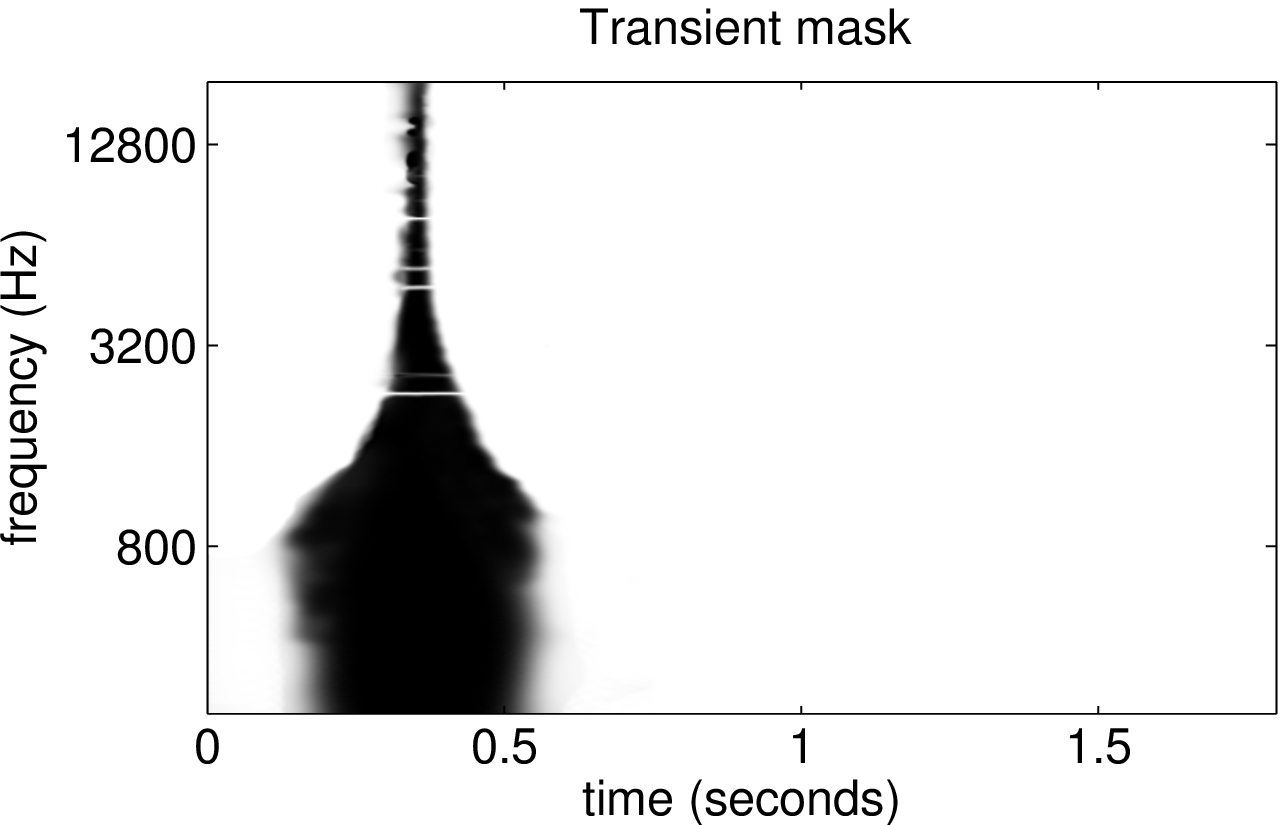}

\vspace{.5cm}
\includegraphics[width=7.5cm]{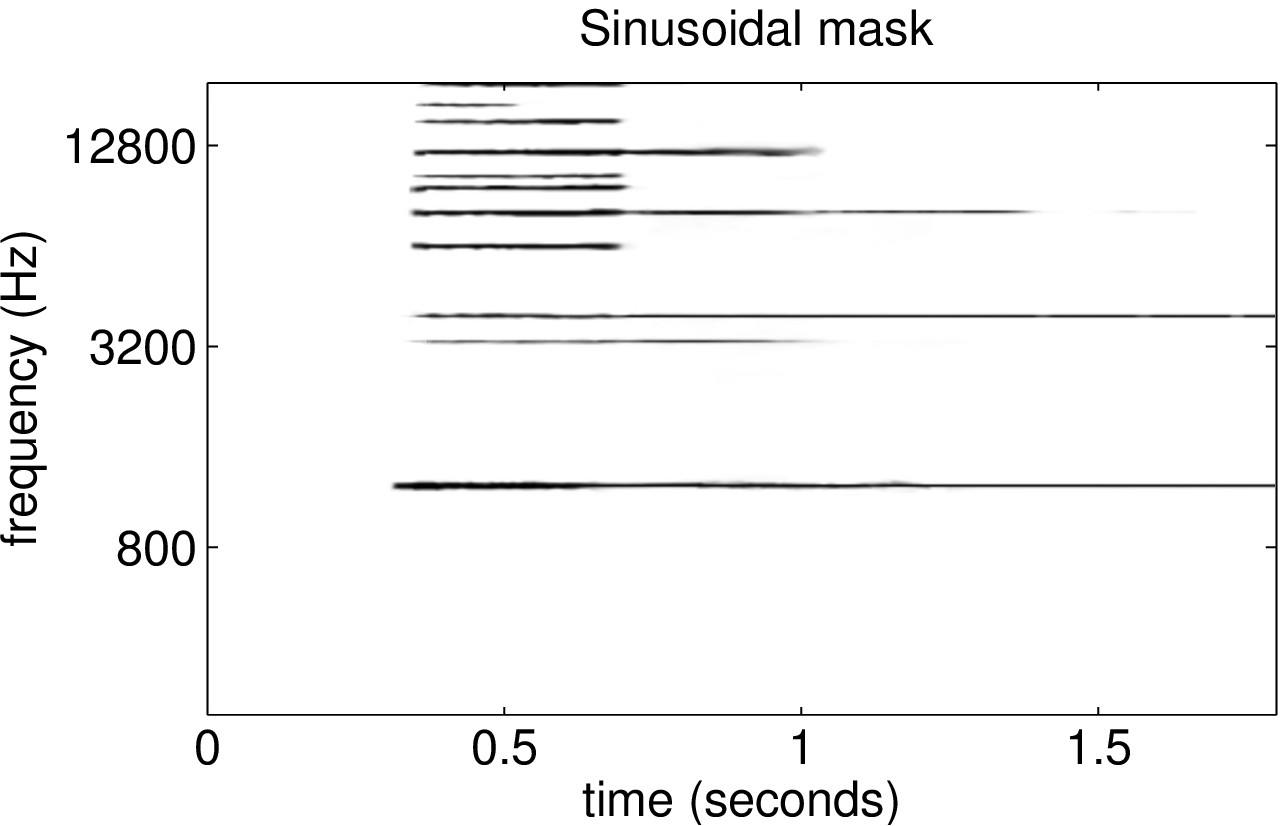}
\end{center}
\caption{Masks for extracting a transient (top) and sinusoidal component (bottom) of the Glockenspiel signal. The gray level plot describes the amplitude of the mask, with black and white representing $1$ and $0$, respectively.}
\label{fig:themasks}
\end{figure}

Figure \ref{fig:themasks} shows masks for isolating a transient part and the corresponding sinusoidal part of a Glockenspiel signal, created using an ordinary image manipulation program. Therein, the layers paradigm has been used to be able to quickly switch on and off the masks in order to accurately adapt them to the CQ-NSGT representation of the audio. An ``inverse mask'' is also constructed for the remainder part of the signal, essentially decomposing the signal into transient, sinusoidal and background portions. The masks have been drawn in the logarithmic domain, to be able to handle the dynamics of the audio. They are linearly scaled in dB units, so that $0$ in the mask corresponds to $10^{-5}$ ($-100$~dB) and $1$ corresponds to $1$ ($0$~dB).

\begin{figure}[t!]
\begin{center}
\includegraphics[width=8.5cm]{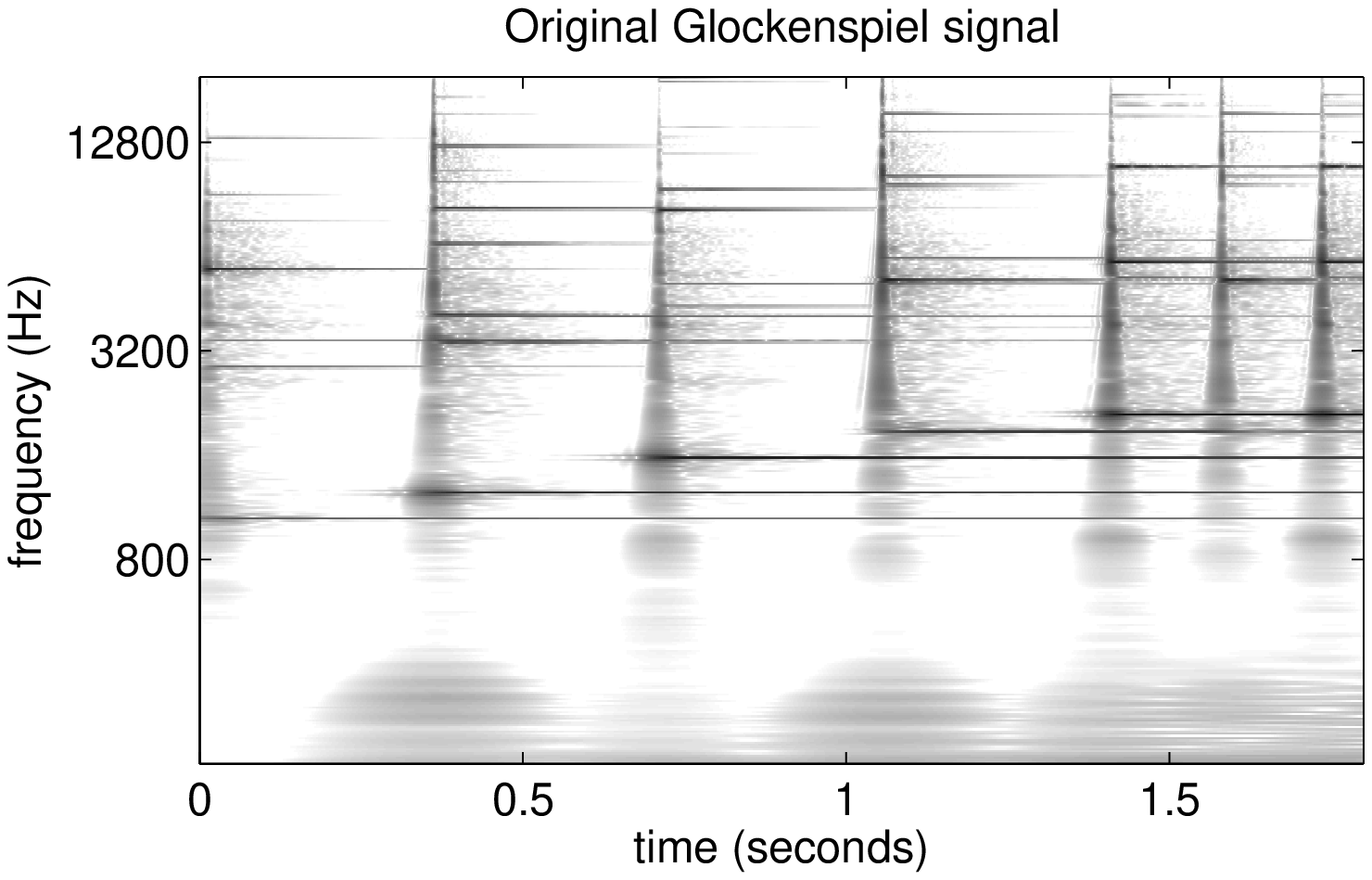}

\vspace{.5cm}
\includegraphics[width=8.5cm]{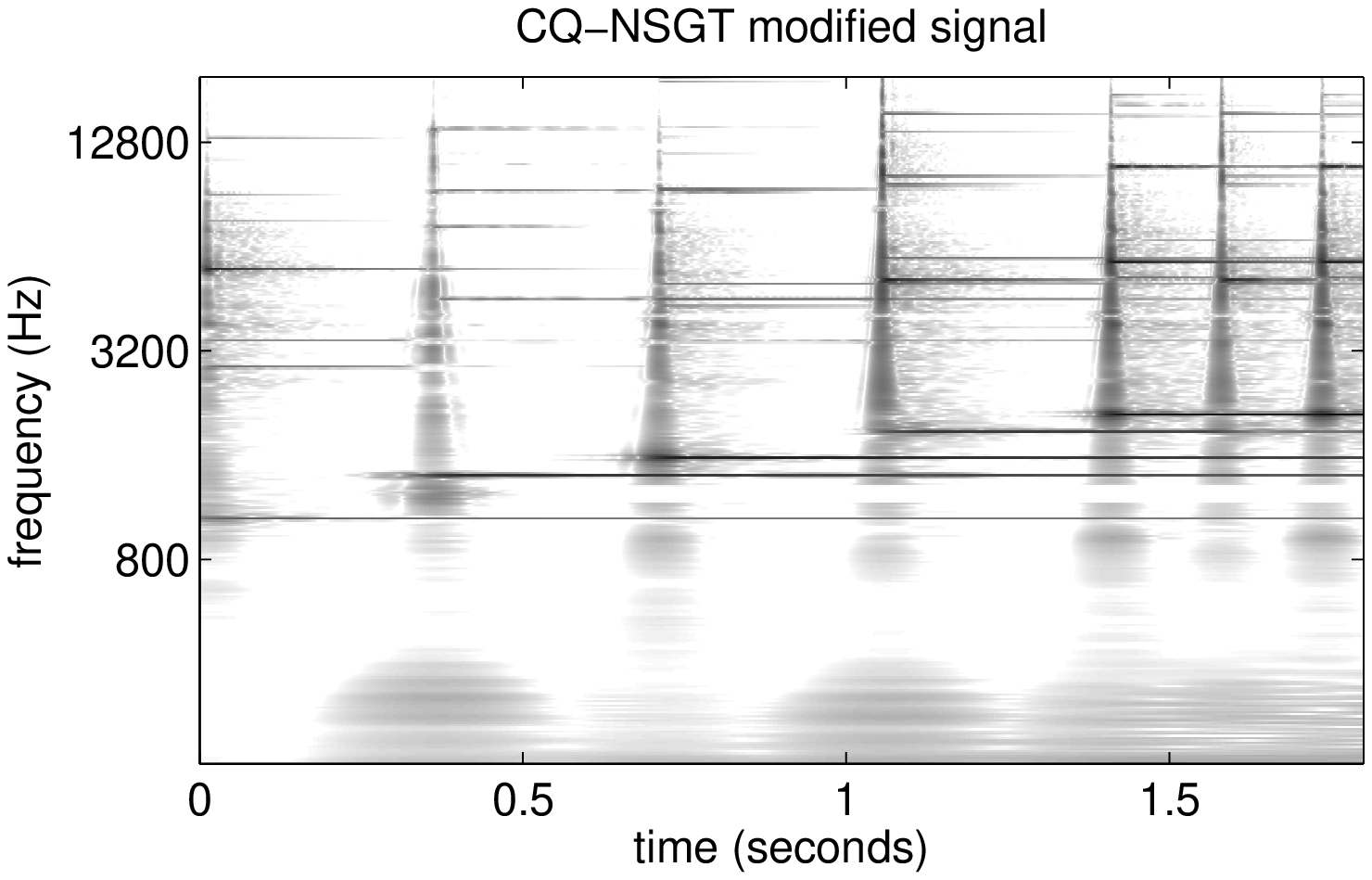}
\end{center}
\caption{CQ-NSGT spectrograms showing an excerpt of the Glockenspiel signal before (top) and after transposition of a component (bottom).}
\label{fig:transpose}
\end{figure}

While keeping the transient part, the isolated sinusoidal component of the signal is transposed upward by $2$ semitones, corresponding to $8$ frequency bins. The transient, the remainder, and the modified sinusoidal coefficients are then added and the inverse transform is applied to obtain the resulting processed signal.
For ease of use, this process is done with a rectangular representation of the slices, 
obtained by choosing $L/a_k$ constant for all frequency bands
which corresponds to a sinc-interpolation of the coefficients. 

Figure \ref{fig:transpose} compares the CQ-NSGT spectrograms of the original and the modified signal, while Figure \ref{fig:transpose2} shows the results for the same experiment using sliCQ transforms with different slice lengths. 
Note that the plots show the spectrogram of the synthesized signal, not the time-frequency coefficients before synthesis. Further, the exact same mask was used for CQ-NSGT and sliCQ transpositions. The sound files for this and other transposition experiments are available at \url{http://www.univie.ac.at/nonstatgab/slicq}. A script for the Python toolbox that executes the experiment, is available on the same page.

For synthesis, performed from modified coefficients, as opposed to mere reconstruction,
an evaluation of the results is a highly non-trivial matter.
This is due to the lack of a properly defined notion of \emph{accuracy} or the existence of a \emph{target signal}, not only for the algorithms presented here, but for any analysis/synthesis based signal processing framework. 
Thus, while the examples in this section should indicate that CQ-NSGT synthesis and sliCQ synthesis can produce results in accordance with intuition, an in-depth treatment of this subject is far beyond the scope of this article.
%
%

\section{Summary and Conclusion}\label{sec:CONCLU}

In this contribution, we have introduced a framework for real-time implementation of an invertible constant-Q transform based on frame theory. The proposed framework allows for straight-forward generalization to other non-linear frequency scales, such as mel- or Bark scale,~cp.~\cite{doevma12}. While real-time processing is possible by means of a preprocessing step, we investigated the possible occurrence of time-aliasing. We provided a numerical evaluation of computation time and quality of approximation of the true NSGT coefficients. 

In analogy to the classical phase vocoder, phase issues have to be addressed, if CQ-transformed coefficients are processed, cp.~\cite{dola97,dola99,Ro03}. While preliminary experiments using the proposed framework for real-life signals were presented, undesired phasing effects, mainly due to the contribution of a signal component to several adjacent filters, will be investigated in detail in future work. Furthermore, future work will consider the efficient realization of adaptivity in both time and frequency by varying the length of the preprocessing windows used for slicing.   

\section{Appendix}

\subsection{Derivation of CQ-NSGT properties}\label{sec:CQprop}
\begin{proof}[Proof of Proposition 1]
  By Algorithm \ref{alg:nsganalysis}, we have 
  \begin{align}
    \lefteqn{c_{n,k} = c_k[n]} \nonumber \\
    & = \sqrt{L/a_k}\frac{1}{L/a_k} \sum_{m=0}^{L/a_k-1}\sum_{l=0}^{a_k-1} (\hat{f}\overline{g_k})[m+l\frac{L}{a_k}]e^{2\pi inm a_k/L} \nonumber \\
    & = \sum_{m=0}^{L/a_k-1}\sum_{l=0}^{a_k-1} (\hat{f}\bd{M}_{na_k}\overline{g_k})[m+l\frac{L}{a_k}]
  \end{align}
  Since $L/a_k \geq L$, only one element of the inner sum above is non-zero, for each $m\in\{0,\ldots,L/a_l-1\}$. It follows that 
  \begin{equation}\label{eq:coeff_nsg}
    c_{n,k} = \langle \hat{f},\bd{M}_{-na_k}g_k \rangle.
  \end{equation}
  Inserting into Algorithm \ref{alg:nsgsynthesis} yields, for all $j\in\{0,\ldots,L-1\}$,
  \begin{align*}
   \hat{\tilde{f}}[j] & = \sum_{k\in I_K} \sum_{n=0}^{L/a_k-1} c_{n,k}e^{-2 \pi inm a_k/L}\widetilde{g_k}[j] \\
   & = \sum_{k\in I_K} \sum_{n=0}^{L/a_k-1} \langle \hat{f},\bd{M}_{-na_k}g_k \rangle \bd{M}_{-na_k}\widetilde{g_k}[j],
  \end{align*}
  the discrete frame synthesis formula. By assumption, $\mathcal{G}(\bd{g},\bd{a})$ and $\mathcal{G}(\tilde{\bd{g}},\bd{a})$ are dual NSG frames and thus
  \begin{equation*}
    \hat{\tilde{f}}[j] = \hat{f}[j], \quad \text{ for all } j\in\{0,\ldots,L-1\}.
  \end{equation*}
  Applying the inverse discrete Fourier transform completes the proof. 
\end{proof}\ \\

\begin{figure}[t!]
\begin{center}
\includegraphics[width=8.5cm]{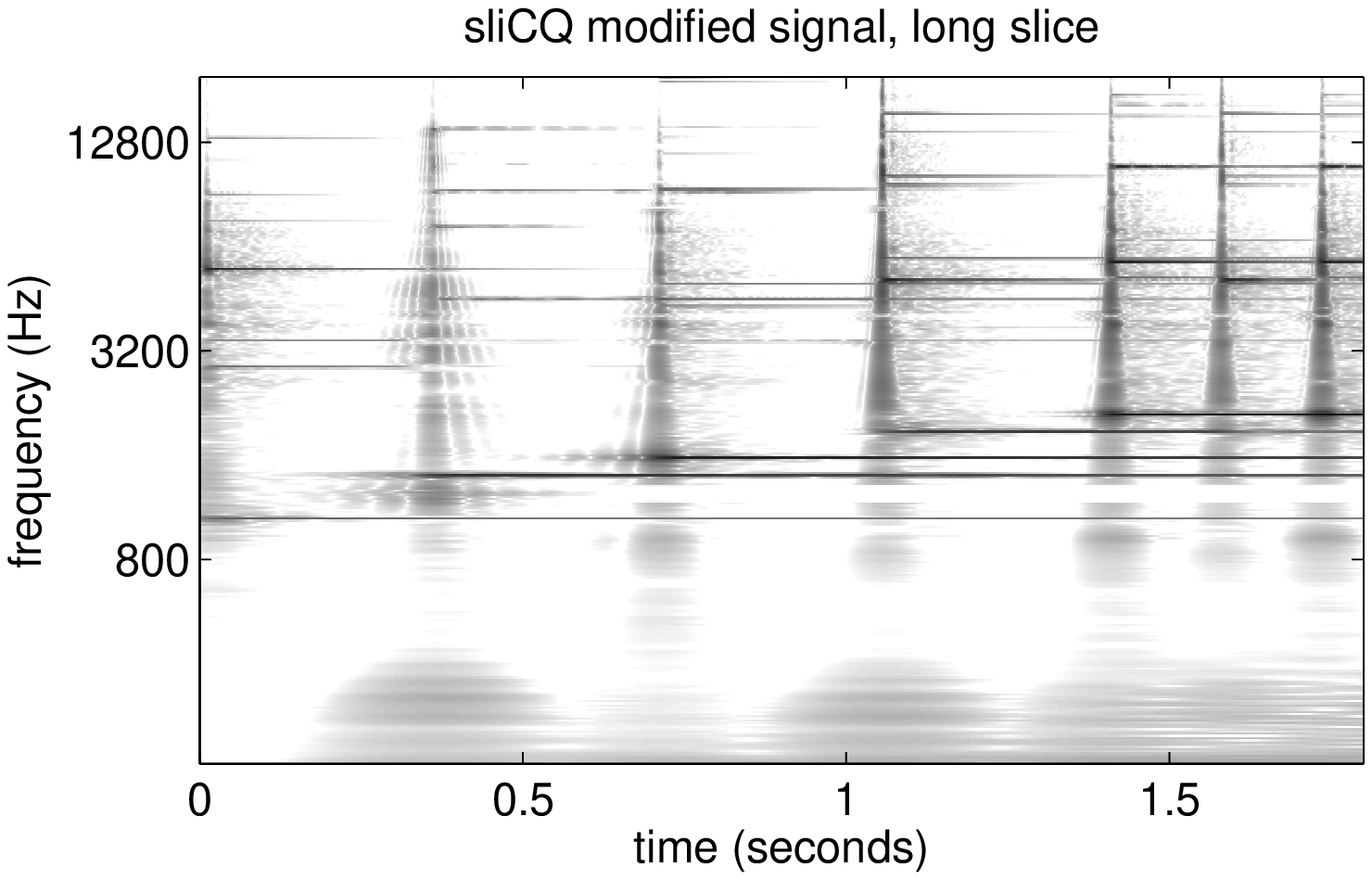}

\vspace{.5cm}
\includegraphics[width=8.5cm]{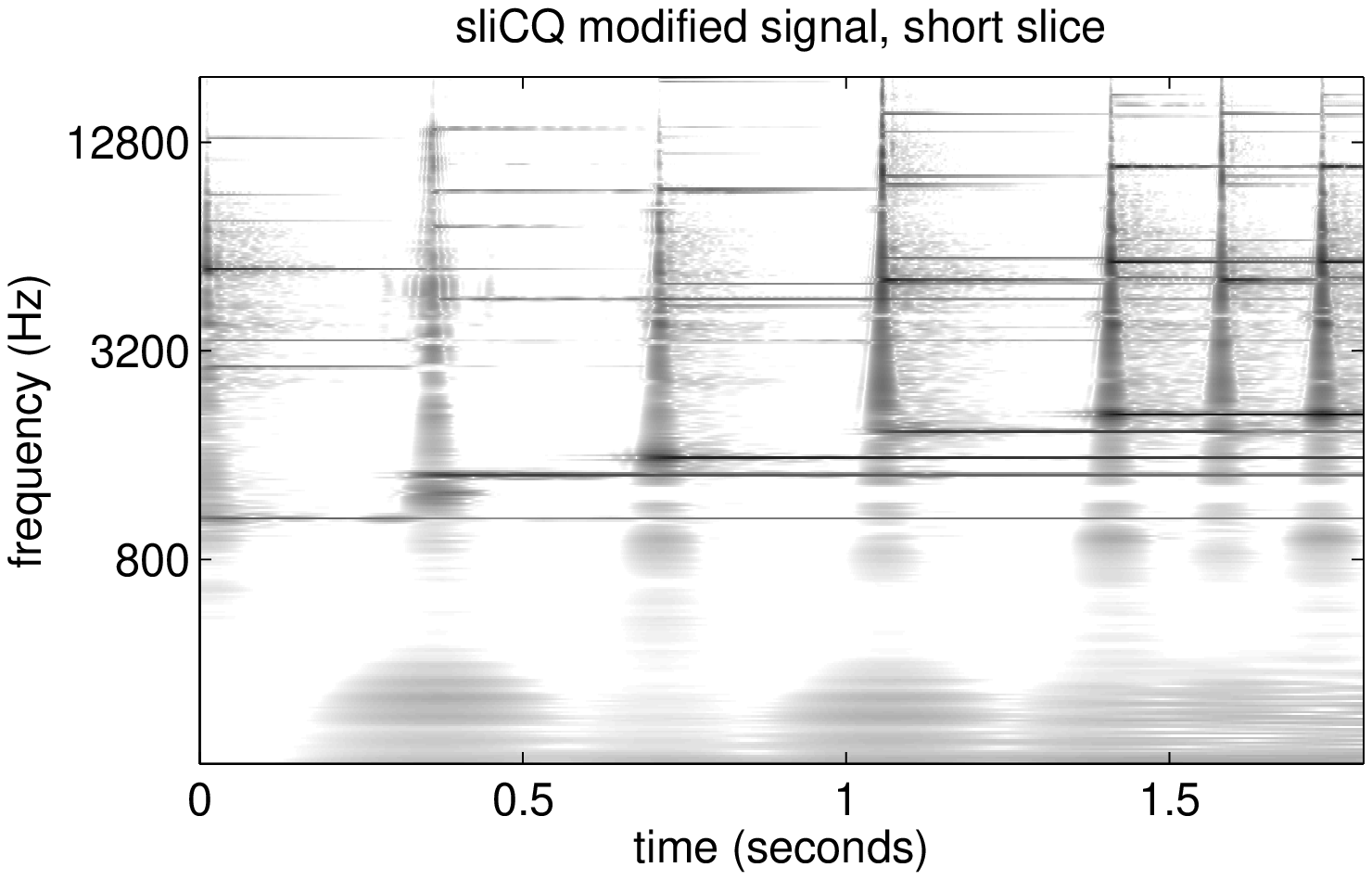}
\end{center}
\caption{sliCQ spectrograms showing an excerpt of the Glockenspiel signal after transposition of a component. The top plot was done with a slice length of $50000$ and a transition area of $20000$ samples, the bottom plot with a slice length of $5000$ and a transition area of $2000$ samples.}
\label{fig:transpose2}
\end{figure}

\begin{proof}[Proof of Proposition 2]
  Denote by $J_k$ an interval of length $L_k$, $L_k$ as in Section \ref{sec:NSGTfr}, containing the support of $g_k$. By assumption
  \begin{equation*}
    0 < \sum_{k\in I_K}|g_k[j]|^2 < \infty,\quad \text{ for all } j = 0,\ldots,L-1
  \end{equation*}
  and $L/a_k \geq L_k = |J_k|$. 
  Note that the frame operator \eqref{eq:frameop} can be written as follows
  \begin{align}
    \bd{S}f[j] 
    & = \sum_{k\in I_K}\sum_{n=0}^{L/a_k-1} \langle f,\bd{M}_{-na}g_k \rangle \bd{M}_{-na}g_k[j] \nonumber \\
    & = \sum_{k\in I_K} \sqrt{\frac{L}{a_k}} \sum_{n=0}^{L/a_k-1} \textbf{IFFT}_{L/a_k}(f\overline{g_k})[n] g_k[j]e^{-2\pi inj a_k/L} \nonumber \\
    & = \sum_{k\in I_K} \frac{L}{a_k}\textbf{FFT}_{L/a_k}(\textbf{IFFT}_{L/a_k}(f\overline{g_k}))[j]g_k[j], \label{eq:frameop2}
  \end{align}
  for all $f\in\CC^L$. Furthermore, with $\chi_{J_k}$ the characteristic function of the interval $J_k$,
  \begin{align*}
    f\overline{g_k} & = \chi_{J_k} \sum_{l=0}^{a_k-1} \bd{T}_{lL/a_k}(f\overline{g_k})\\
    & = \chi_{J_k}\textbf{FFT}_{L/a_k}(\textbf{IFFT}_{L/a_k}(f\overline{g_k}))
  \end{align*}
  and, obviously, $g_k = \chi_{J_k}g_k$. Inserting into \eqref{eq:frameop2} yields
  \begin{align}
    \bd{S}f[j] & = \sum_{k\in I_K} \frac{L}{a_k} (f\overline{g_k})[j]g_k[j] \nonumber \\
               & = f[j]\sum_{k\in I_K} \frac{L}{a_k}|g_k|^2[j].
  \end{align}
  With the sum bounded above and below, the inverse frame operator can be written as
  \begin{equation}\label{eq:invframeop1}
    \bd{S}^{-1}f[j]  = f[j]\left(\sum_{k\in I_K} \frac{L}{a_k}|g_k|^2[j]\right)^{-1}, \text{ for all } f\in\CC^L.
  \end{equation}
  Since the elements of the canonical dual frame are given by \eqref{eq:candual}, this completes the proof.
\end{proof}

\subsection{Derivation of sliCQ properties}\label{sec:sliCQprop}
\begin{proof}[Proof of Proposition~\ref{Prop:sliPerfRec}]
According to Proposition~\ref{pro:reconst}, $\tilde{f}^m$, the output of \textbf{iCQ-NSGT} in Step~9 of Algorithm~\ref{alg:slicqsyn}
satisfies to $f^m[j] = (f\cdot\bd{T}_{mN}h_0)[j+(m-1)N]$. Since $\sum_m \bd{T}_{mN} \left(h_0\overline{\tilde{h}_0}\right) \equiv 1$ holds,
\begin{equation*}
  \tilde{f} = \sum_m (f\cdot\bd{T}_{mN}h_0)\bd{T}_{mN}\overline{\tilde{h}_0}  = f\cdot\sum_m \bd{T}_{mN} \left(h_0\overline{\tilde{h}_0}\right) = f
\end{equation*}
follows.
\end{proof}\ \\

\begin{proof}[Proof of Proposition~\ref{pro:CQ_sliCQ} ]
Since $g_k$ is obtained by sampling $g^\mathcal{L}_k$ with sampling period $L/2N$, the (inverse) Fourier transform  $\widecheck{g_k}$ of $g_k$ is given by periodization of $g^\mathcal{L}_k$ as follows:
\begin{equation}
	\widecheck{g_k} [l] = \sum_{j = 0}^{\frac{L}{2N}-1} \widecheck{g^\mathcal{L}_k}[l +j\cdot 2N].
\end{equation}
Recall from \eqref{eq:CQ_coeff} that  the CQ-NSGT coefficients of $f$ with respect to $\mathcal{G}(\bd{g}^\mathcal{L},\bd{a})$ are given by
$c_{n,k} = \langle f,\mathbf{T}_{n a_k}\widecheck{g^\mathcal{L}_k} \rangle$, while 
the CQ-NSGT coefficients $c^m $ of $f^m$ are, for $m =0,\ldots,L/N-1,\,n^s= 0,\ldots, \frac{2N}{a_k}-1$ and $k \in I_K$ 
  \begin{align}\label{eq:SLICQtime}
  c^m_{n^s,k} &= \langle \widehat{f^m}, g_{n^s,k} \rangle = \langle \widehat{f^m}, \bd{M}_{-n^s a_k}g_k \rangle \nonumber \\
         & = \langle f^m, \bd{T}_{n^s a_k}\widecheck{g_k} \rangle \nonumber \\
         & = \left\langle f, h_m \sum_{j = 0}^{\frac{L}{2N}-1}  \bd{T}_{n^s a_k+(m-1+2j)N} \widecheck{g^\mathcal{L}_k} \right\rangle,
  \end{align}
  where the final inner product is taken over $\CC^L$.  
Observe that every $n = 0, \ldots, \frac{L}{a_k}-1$ can be written as $n = m\frac{N}{a_k}+ n^s$ with $n^s$ from $0,\ldots, \frac{N}{a_k}-1$ and thus
  \begin{align}
s^0_{n,k}+s^1_{n,k}& = c^{m}_{n^s+N/a_k,k} + c^{m+1}_{n^s,k} \nonumber \\
	& = \left\langle f, (h_m+h_{m+1})\sum_{j= 0}^{\frac{L}{2N}-1} \bd{T}_{n^s a_k+(m+2j)N} \widecheck{g^\mathcal{L}_k} \right\rangle \nonumber\\ 
        & = \left\langle f,\bd{T}_{n^s a_k+mN}\widecheck{g^\mathcal{L}_k} \right\rangle+R[n] \nonumber\\ 
        & = \left\langle f,\bd{T}_{ a_k(\frac{mN}{a_k}+n^s)}\widecheck{g^\mathcal{L}_k} \right\rangle+R[n] \nonumber\\ 
        &= c_{n,k} + R[n].
  \end{align}
  Here,
   \begin{align}R[n] =  &\left\langle f, (1-h_{m}-h_{m+1})\bd{T}_{n^s a_k+mN}\widecheck{g^\mathcal{L}_k}\right\rangle \nonumber \\
   +&\left\langle  (h_{m}+h_{m+1})\sum_{j= 1}^{\frac{L}{2N}-1} \bd{T}_{n^s a_k+(m+2j)N} \widecheck{g^\mathcal{L}_k}\right\rangle.\end{align}
Hence $s^0_{n,k}+s^1_{n,k}-c_{n,k} = R[n]$. The result follows from Cauchy-Schwartz' inequality, applied to the case $m=0$, observing independence from $m$.
\end{proof}

\section*{Acknowledgment}
This research was supported by the  WWTF project Audio-Miner (MA09-024), the Austrian Science Fund (FWF):[T384-N13] and the 
EU FET Open grant UNLocX (255931). The authors wish to thank the reviewers for their extremely helpful and constructive  remarks
on the first version of the manuscript.

\end{document}